\documentclass[11pt]{amsart}

\usepackage{amssymb, amsmath}
\usepackage{amsthm}
\usepackage{amscd}
\usepackage{graphicx}
\usepackage{tikz-cd}

\newtheorem{theorem}{Theorem}[section]
\newtheorem{lemma}[theorem]{Lemma}

\newtheorem{corollary}[theorem]{Corollary}
\newtheorem{proposition}[theorem]{Proposition}

\theoremstyle{definition}

\newtheorem*{remark}{Remark}
\newtheorem*{acknowledgement}{Acknowledgement}

\newtheorem*{conj}{Conjecture}

\def\beqa{\begin{eqnarray}}
\def\eeqa{\end{eqnarray}}
\def\beqa{\begin{eqnarray}}
\def\eeqa{\end{eqnarray}}

\DeclareMathOperator{\Ker}{Ker}

\DeclareMathOperator{\Stab}{Stab}
\DeclareMathOperator{\Isom}{Isom}
\DeclareMathOperator{\GL}{GL}
\DeclareMathOperator{\SL}{SL}
\DeclareMathOperator{\SO}{SO}
\DeclareMathOperator{\SU}{SU}

\DeclareMathOperator{\diam}{diam}

\DeclareMathOperator{\diag}{diag}
\DeclareMathOperator{\rank}{rank}

\def\R{\mathbb{R}}

\def\sl{\mathfrak{sl}}

\input xy
\xyoption{all}

\begin{document}

\title{Geometric cycles and bounded cohomology for a cocompact lattice in $\SL_n(\mathbb R)$}

\author[Shi Wang]{Shi Wang}
\address{Max Planck Institut f\"ur Mathematik\\
Bonn\\
Germany}
\email{shiwang.math@gmail.com}

\begin{abstract}
We show there exists a closed locally symmetric manifold $M$ modeled on $\SL_n(\mathbb R)/\SO(n)$, and a non-trivial homology class in degree $\dim(M)-\rank(M)$ represented by a totally geodesic submanifold that contains a circle factor.
As a result, the comparison map $c^k:H_b^k(M,\mathbb R)\rightarrow H^k(M,\mathbb R)$ is not surjective in degree $k=\dim(M)-\rank(M)$. This provides a counterpart to a result of \cite{LW19} which states that $c^k$ is always surjective in degree $k\geq \dim(M)-\rank(M)+2$.
\end{abstract}

\maketitle


\section{Introduction}

Let $M$ be a connected, closed, oriented topological manifold. For each singular homology class $\alpha\in H_k(M,\mathbb R)$, one can associate to it a semi-norm which measures how efficiently $\alpha$ can be represented by a linear combination of simplices. More precisely, we define the \emph{Gromov norm}
\[||\alpha||_1:=\inf\{\sum_{i=1}^l|a_i|\;:\;\sum_{i=1}^la_i\sigma_i\textit{ is a cycle representing }\alpha\textit{ in }H_k(M,\mathbb R)\}.\]
	In particular, the Gromov norm of the fundamental class $[M]$ is called the \emph{simplicial volume} of $M$, denoted by $||M||$.

One interesting aspect of this topological invariant is that it reflects certain geometric properties of the underlying manifold. It is shown that the bounded cohomology of $M$ vanishes if the fundamental group $\pi_1(M)$ is amenable \cite{Gro82}, in which case the Gromov norm vanishes in all degrees according to the duality principle. This includes, for example, positively curved manifolds, flat manifolds, and most generally manifolds with nonnegative Ricci curvature. On the other hand, for a negatively curved manifold, all non-trivial $k$-classes have positive Gromov norm when $k\geq 2$ \cite{IY82}. It remains a mysterious question for manifolds of nonpositive curvature whether the Gromov norm is positive or zero. Intuitively speaking, if the homology class resembles more of a negatively curved feature, or that overall the negative curvature dominates the zero curvature, then the Gromov norm is positive. Some efforts have been made towards this direction in \cite{CW19, CW20}.

In the case of irreducible locally symmetric manifolds of noncompact type, it is shown that any non-trivial $k$-class has positive Gromov norm when $k\geq \dim(M)-\rank(M)+2$ \cite{LW19, LS06, Buc07}, and this has been generalized in \cite{Wang20}. All these results are based on the original straightening method introduced in \cite{Thu97}, which boil down to showing certain straightened $k$-simplices have uniformly bounded volume. However, as noted in \cite{LW19}, this in general fails when $k\leq \dim(M)-\rank(M)$ if $M$ is modeled on $\SL_n(\mathbb{R})/\SO(n)$. This suggests that there could be examples of nontrivial classes with zero Gromov norm in the corresponding degrees. In the present paper, we give an explicit construction of a non-trivial homology class in degree $k=\dim(M)-\rank(M)$, which is represented by a submanifold with a circle factor.

\begin{theorem}\label{thm:main}
For every integer $n\geq 3$, there exists a torsion-free cocompact lattice $\Gamma<\SL_n(\mathbb R)$, constructed as a congruence subgroup of the explicit example in Section \ref{sec:lattice}, and two non-trivial homology classes $\alpha\in H_{k}(M,\mathbb Q)$ and $\beta\in H_{n-1}(M,\mathbb Q)$ in the associated locally symmetric manifold $M=\Gamma\backslash \SL_n(\mathbb R)/\SO(n)$ where $k=\dim(M)-\rank(M)=n(n-1)/2$, such that 
\begin{enumerate}
\item $\alpha$ is represented by a totally geodesic submanifold $H$ covered by $\SL_{n-1}(\mathbb R)/\SO(n-1)\times \mathbb R$,
\item $\beta$ is represented by a flat $(n-1)$-torus $T$, and
\item the intersection number $i(H,T)\neq 0$.
\end{enumerate}
\end{theorem}

This immediately implies that the Gromov norm of the homology class $\alpha$ is zero. Therefore, we obtain the following corollary.

\begin{corollary}\label{cor:main} For every integer $n\geq 3$, there exists a torsion-free cocompact lattice $\Gamma<\SL_n(\mathbb R)$, and a nontrivial homology class $\alpha\in H_k(M,\mathbb R)$ in the associated locally symmetric manifold $M=\Gamma\backslash \SL_n(\mathbb R)/\SO(n)$ where $k=\dim(M)-\rank(M)=n(n-1)/2$, such that the Gromov norm $||\alpha||_1=0$. Equivalently, the comparison map
$H_b^k(\Gamma,\mathbb R)\rightarrow H^k(\Gamma,\mathbb R)$ is \emph{not} surjective in degree $k=n(n-1)/2$.
\end{corollary}

Our construction and proof are very similar to the general approach of the $\SL_n(\mathbb Z)$ case in the paper of Avramidi--Nguy$\tilde{\hat{\mathrm{e}}}$n-Phan \cite{ANP15}, and their simplified proof in \cite{ANP20}. For the purpose of making $\Gamma$ cocompact, we need to pass onto a number field, and the resulting extra number theoretic complexity makes the proof more difficult. Besides, geometric cycles in other type of symmetric spaces are constructed in \cite{MR80, Bena21} via different methods.

\subsection*{Organization of the paper} In Section \ref{sec:prelim}, we provide some results in number theory focusing on the number field $\mathbb Q[\sqrt[4]2]$ which will be used later in the proof. We give in Section \ref{sec:construction} the explicit constructions of the lattices as well as the homology cycles, we also describe the intersections of the cycles. In Section \ref{sec:proof}, we investigate further the intersections under congruence covers and prove Theorem \ref{thm:main}. In the last Section \ref{sec:Gromov norm}, we discuss the applications to the study of Gromov norm and bounded cohomology.


\begin{acknowledgement} The author would like to thank Grigori Avramidi and Tam Nguy$\tilde{\hat{\mathrm{e}}}$n-Phan for their helpful discussions, through which this work was motivated and initiated. I am very grateful for their careful explanations on the general approach and simplified proof of the $\SL_n(\mathbb Z)$ case in \cite{ANP15, ANP20}. I would also like to thank Grigori Avramidi for the innumerous help on the number theoretic arguments including the application of Grunwald-Wang theorem. Thanks to Pengyu Yang for pointing out the reference \cite{Conrad}. Also, I would like to thank the anonymous referee for helpful suggestions. The author is supported by the Max Planck Institut f\"ur Mathematik while this work was completed.
\end{acknowledgement}

\section{Preliminary}\label{sec:prelim} In this section, we give a brief introduction to some basic knowledge of algebraic number theory that will be used later in our proofs. Since we only focus on the field $\mathbb Q[\sqrt[4]2]$, most of the arguments become rather elementary and tedious. We suggest the readers who are familiar with the contents skip ahead to the next section and check back later when they are actually used.

\subsection*{Rings of integers}
A subfield $L\subset \mathbb C$ is called a \emph{number field} if it is a finite extension of $\mathbb Q$. In this paper, we are mainly concerned with the case $L = \mathbb{Q}[\sqrt[4]{2}]$. Let $\mathcal O\subset \mathbb C$ be the set of all \emph{algebraic integers}, that is, the set of roots of monic polynomials in integer coefficients. We denote $\mathcal O_L=\mathcal O\cap L$ the \emph{ring of integers} in $L$. 

\begin{lemma}\label{lem:integer}\cite[Theorem 3.1]{Conrad}
The ring of integers of $\mathbb{Q}[\sqrt[4]{2}]$ is $\mathbb{Z}[\sqrt[4]{2}]$. 
\end{lemma}

\subsection*{Units} An element of $\mathcal{O}_L$ is called a \emph{unit} if it has a multiplicative inverse. Denote $\mathcal{U}$ the set of units in $\mathcal{O}_L$. The Dirichlet's unit theorem states that the group of units $\mathcal U\subset \mathcal O_L$ is finitely generated and the rank is equal to $r_1+r_2-1$, where $r_1$ is the number of real embeddings and $r_2$ is the number of conjugate pairs of complex embeddings of $L$ into $\mathbb C$. Now if we restrict our attention to the case $L=\mathbb{Q}[\sqrt[4]{2}]$, and denote $\tau$ the Galois automorphism that sends $\sqrt[4]{2}$ to $-\sqrt[4]{2}$, then $\mathcal U$ has rank two, and the following subset of units
\[\mathcal U_0=\{u\in \mathcal U\;|\;\tau(u)\cdot u=1\}\]
is rank one, which can be explicitly given by the following lemma.

\begin{lemma} \label{lem:fundamental unit}
We have $\mathcal U_0=\{\pm u_0^k\;|\;k\in \mathbb Z\}$, where $u_0=(3+2\sqrt 2)+(2+2\sqrt 2)\sqrt[4]{2}$.
\end{lemma}

\begin{proof}
First, we can check $\tau(u_0)\cdot u_0=1$, so $u_0$ is a unit. This implies $\{\pm u_0^k\;|\;k\in \mathbb Z\}\subset \mathcal U_0$. Secondly, $\sqrt 2-1$ is an element of infinite order in $\mathcal U$ which does not satisfy $\tau(u)\cdot u=1$, so the rank of $\mathcal U_0$ must be $1$. Notice that the only torsion elements in $\mathcal U$ (hence also in $\mathcal U_0$) are $\pm 1$, so it remains to show $u_0$ is a primitive generator, that is, not a power of any elements in $\mathcal U_0$.

Assume not, there exists $v\in \mathcal U_0$, and an integer $k>1$ such that $u_0=v^k$. By possibly switching the sign of $v$, we may assume $v>0$, so we have $1<v<u_0$. According to Lemma \ref{lem:integer}, we can write $v=(\alpha_1+\beta_1\sqrt 2)+(\alpha_2+\beta_2\sqrt 2)\sqrt[4]{2}$ where $\alpha_1,\alpha_2,\beta_1,\beta_2\in \mathbb Z$, the assumption $\tau(v)\cdot v=1$ now gives
$$(\alpha_1+\beta_1\sqrt 2)^2-(\alpha_2+\beta_2\sqrt 2)^2\sqrt 2=1.$$
Apply the Galois automorphism $\sigma:\sqrt 2\rightarrow -\sqrt 2$ on the above equation, we obtain
$$(\alpha_1-\beta_1\sqrt 2)^2+(\alpha_2-\beta_2\sqrt 2)^2\sqrt 2=1.$$
Thus $|\alpha_1-\beta_1\sqrt 2|\leq 1$ and $|\alpha_2-\beta_2\sqrt 2|<1$. If $\beta_2=0$, then $\alpha_2=0$. It follows that $\alpha_1-\beta_1 \sqrt 2=\pm 1$, which forces $\beta_1=0$ and $\alpha_1=\pm 1$, a contradiction. Similarly if $\beta_1=0$, it forces $\alpha_1=\pm 1$ and $\alpha_2=\beta_2=0$, which is again impossible. This implies that $\alpha_1, \beta_1$ must have the same sign and so do $\alpha_2, \beta_2$. Furthermore, $v>1$ implies $0<\tau(v)<1<v$, so $\alpha_1,\alpha_2,\beta_1,\beta_2>0$. Now since $\alpha_1,\alpha_2,\beta_1,\beta_2$ are all positive integers, the least possible $v$ is $(1+\sqrt 2)+(1+\sqrt 2)\sqrt[4]2\approx 5.285$ which contradicts with $12>u_0=v^k\geq v^2$. Therefore, $u_0$ is a primitive generator of $\mathcal U_0$, which completes the proof.
\end{proof}

\subsection*{The Grunwald-Wang theorem} The (Hasse) local-global principle plays a very important role in number theory. The idea comes down to that finding rational solutions for certain equations is sometimes equivalent to finding $p$-adic solutions for all prime $p$. The Grunwald-Wang theorem is a version of such principle for the special family of polynomial equations $x^m=a$.

Let $L$ be a number field, we denote by $\zeta_r$ a primitive $2^r$-th root of unity, and $\eta_r=\zeta_r+\zeta_r^{-1}$. Let $s\geq 2$ be the integer such that $\eta_s\in L$ but $\eta_{s+1}\notin L$. 

\begin{theorem}[Grunwald-Wang]\cite[Chapter 10]{Artin-Tate}
Let $m$ be an integer, $S$ be a finite set of primes and $P(m,S)$ the group of all $a\in L$ such that $a\in L_p^m$ for all $p\notin S$. Then $P(m,S)= L^m$ except when the following conditions are all satisfied (which is referred to as the special case):
\begin{enumerate}
\item $-1$, $2+\eta_s$ and $-(2+\eta_s)$ are non-squares in $ L$,
\item $m=2^tm'$ where $m'$ is odd and $t>s$.
\item $S_0\subset S$, where $S_0$ is the set of primes $p|2$ where $-1$, $2+\eta_s$ and $-(2+\eta_s)$ are non-squares in $ L_p$.
\end{enumerate} 
\end{theorem}

In the case $ L=\mathbb Q[\sqrt[4]2]$, we will see in the following lemma that the exceptional set $S_0$ is non-empty, therefore the special case must have $S\neq \emptyset$.

\begin{lemma} If $ L= \mathbb Q[\sqrt[4]2]$, then $\sqrt[4]2\in S_0$.
\end{lemma}

\begin{proof}
It is clear that $s=3$ and $\eta_s=\sqrt 2$. So we need to show $-1, \pm(2+\sqrt 2)$ are non-squares in the $p$-adic field $L_{\sqrt[4]2}$. As $-1, \pm(2+\sqrt 2)$ are all $p$-adic integers, it suffices to show they are not the square of any element in $\mathcal O_{\sqrt[4]2}$ where $\mathcal O=\mathbb Z[\sqrt[4]2]$. Note that $\mathcal O_{\sqrt[4]2}$ by definition is the inverse limit of the rings $\mathcal O/{(\sqrt[4]2)^k}$, so if a $\sqrt[4]2$-adic integer is a square in $\mathcal O_{\sqrt[4]2}$, then it is also a square in $\mathcal O/{(\sqrt[4]2)^k}$ for any natural number $k$, and in particular it is a square in $\mathcal O/{(4)}\cong \mathbb Z[x]/(x^4-2,4)$. We show that $-1, \pm(2+\sqrt 2)$ are non-squares in $\mathcal O/{(4)}$ thus the lemma follows.

Take the mod $2$ ring homomorphism $\varphi:\mathbb Z[x]/(x^4-2,4)\rightarrow \mathbb Z[x]/(x^4,2)$, and we put a bar to indicate the image under this homomorphism (for example, $\bar\alpha(x)=\varphi(\alpha(x))$). We can further identify the image $\mathbb Z[x]/(x^4,2)$ as the set of integral polynomials with $\{0,1\}$-coefficient (together with the relation $x^4=0$), which can be viewed as a subset in $\mathbb Z[x]/(x^4-2,4)$. We claim that $\alpha(x)=\beta^2(x)$ in $\mathbb Z[x]/(x^4-2,4)$ if and only if $\alpha(x)=\bar\beta^2(x)$ in $\mathbb Z[x]/(x^4-2,4)$. Indeed, $\ker(\varphi)=(2)$ implies that $\beta(x)$ and $\bar\beta(x)$ differs by $2\delta(x)$ for some $\delta(x)\in \mathbb Z[x]/(x^4-2,4)$. So we can check $\beta^2(x)=(\bar\beta(x)+2\delta(x))^2=\bar\beta^2(x)+4\bar\beta(x)\delta(x)+4\delta^2(x)=\bar\beta^2(x)$ in $\mathbb Z[x]/(x^4-2,4)$. Use this fact, we can deduce to much less possibilities in checking a square.

If $-1$ were a square $\beta_1^2(x)$ in $\mathbb Z[x]/(x^4-2,4)$, then under the image of $\varphi$, $1=\bar\beta_1^2(x)$ in $\mathbb Z[x]/(x^4,2)$. Hence the only possibilities for $\bar\beta_1(x)$ are $1,1+x^2, 1+x^2+x^3$. But none of the squares is $-1$ in $\mathbb Z[x]/(x^4-2,4)$. Similarly, if $\pm(2+x^2)$ (which represents $\pm(2+\sqrt 2)$ in $\mathcal O/{(4)}$) were a square $\beta_2^2(x)$ in $\mathbb Z[x]/(x^4-2,4)$, then under the image of $\varphi$, $x^2=\bar\beta_2^2(x)$ in $\mathbb Z[x]/(x^4,2)$. Hence the only possibilities for $\bar\beta_2(x)$ are $x,x+x^2, x+x^2+x^3$, but again none of the squares is $\pm(2+x^2)$ in $\mathbb Z[x]/(x^4-2,4)$. This shows that $-1, \pm(2+\sqrt 2)$ are non-squares in $\mathcal O/{(4)}$ and therefore they are non-squares in $L_{\sqrt[4]2}$.
\end{proof}

Now we apply the Grunwald-Wang theorem to the case $L=\mathbb Q[\sqrt[4]2]$ and $S=\emptyset$, and from the above lemma we see that the exceptional condition $(3)$ is not satisfied. Thus we have $P(m,S)= L^m$ for any $m$. In other words, the following corollary holds. We will only use this version in the proof later.

\begin{corollary}\label{cor:Grunwald-Wang}
 Let $L=\mathbb Q[\sqrt[4]2]$, and $m$ be any positive integer. If $a$ is an element in $L$ such that the equation $x^m=a$ has a solution in $L_p$ for every prime ideal $p$, then $x^m=a$ also has a solution in $L$.
\end{corollary}

\section{The constructions}\label{sec:construction}

We start with the following explicit algebraic construction of cocompact lattices in $\SL_n(\mathbb R)$. They are higher dimensional versions of \cite[Example 6.3.2]{Mor15}, which is also discussed in much details in \cite{Ben09}.

\subsection{A cocompact lattice}\label{sec:lattice}
Let $L=\mathbb{Q}[\sqrt[4]{2}]$ be the number field and $\mathcal{O}_L=\mathbb{Z}[\sqrt[4]{2}]$ be the ring of integers of $L$. Denote $\tau \colon L \rightarrow L$ the field automorphism that sends $\sqrt[4]{2} \mapsto -\sqrt[4]{2}$. Set the quadratic form

\[
D_n= \left(\begin{array}{cccc}
-1 &  &  &  \\ 
 & \sqrt 2 &  &  \\ 
 &  & \ddots &  \\ 
 &  &  & \sqrt 2
\end{array}\right).
\]
Then the following subgroup 
\[\Gamma_n \colon = \{g \in \SL_n(\mathcal{O}_L) \;|\;\tau(g^T)D_n g = D_n\}\]
is a cocompact lattice in $\SL_n(\mathbb{R})$ \cite[Corollary 2.18]{Ben09}.

\subsection*{Congruence subgroups}

By Selberg's lemma, there exists a torsion free, finite index subgroup $\Gamma <\Gamma_n$, which corresponds to a closed locally symmetric manifold $M= \Gamma\backslash X_n$, where $X_n=\SL_n(\mathbb R)/\SO(n)$ is the associated symmetric space. 

Arithmetic lattices have many finite index subgroups that inherits certain algebraic structures. Indeed, for any prime ideal $p\subset \mathcal O_L$ and for any positive integer $k$, we have the natural ``mod $p^k$'' homomorphism
\[\varphi_{p^k} \colon \SL_n(\mathcal O_L) \longrightarrow \SL_n(\mathcal O_L/p^k)\]
where the target is a finite group. So by the group isomorphism theorems, the group
\[ \Gamma(p^k) = \Gamma \cap \Ker(\varphi_{p^k})\]
is a finite index normal subgroup of $\Gamma$. Geometrically, the corresponding closed manifold
\[ M'= \Gamma(p^k)\backslash X_n \]
is a finite degree cover of $M$.


\subsection{The cycles}\label{sec:cycles}
We are going to construct two cycles on $M$ which are represented by two closed totally geodesic submanifolds in the complementary dimensions.

\subsection*{Flats and tori} The diagonal matrices in $\SL_n(\mathbb R)$ corresponds to a totally geodesic maximal flat $F\subset X_n$ of dimension $(n-1)$. Let $u_0\in \mathcal O_L$ be the unit as in Lemma \ref{lem:fundamental unit} which satisfies $\tau(u_0)\cdot u_0=1$. Then it follows that the subset
\[
A_n'=\{ \left(\begin{array}{cccc}
u_0^{k_1}  &  &  &  \\ 
 &  u_0^{k_2}&  &  \\ 

 &  &  \ddots &  \\ 
 &  &  & u_0^{k_n}
\end{array}\right)\colon \sum_{i=1}^n k_i=0,\; k_i \in \mathbb{Z}
\}
\]
is a subgroup of $\Gamma_n$ that is isomorphic to $\mathbb{Z}^{n-1}$. Intersecting with $\Gamma$, we obtain a finite index subgroup $A_n<A_n'$ which is also abstractly isomorphic to $\mathbb{Z}^{n-1}$. It acts cocompactly by isometry on the flat $F$, thus passing down to the quotient, the natural inclusion $F\subset X_n$ induces a totally geodesic embedding $A_n\backslash F\subset \Gamma\backslash X_n$. In particular, the image is a closed isometrically embedded $(n-1)$-torus in $\Gamma\backslash X_n$, which we denote by $T$.

\subsection*{Totally geodesic $X_{n-1}\times \mathbb{R}$} The following block diagonal matrices
\[
\left(\begin{array}{c|c}
 \dfrac{1}{t}\cdot\SL_{n-1}(\mathbb R) & 0 \\
 \hline 
 0 &  t^{(n-1)} \\ 
\end{array}\right),t>0\\
\]
form a Lie subgroup in $\SL_n(\mathbb R)$ which is isomorphic to $\SL_{n-1}(\mathbb R)\times \mathbb R$. It corresponds to a totally geodesic submanifold $X_{n-1}\times \mathbb R\subset X_n$ whose dimension equals $\dim(X_n)-\rank(X_n)$. It is clear that the following subset
\[
B_n'=\{ \left(\begin{array}{c|c}
 \dfrac{1}{u_0^{k}}\cdot\Gamma_{n-1} & 0 \\
 \hline 
 0 &  u_0^{k(n-1)} \\ 
\end{array}\right),k\in \mathbb Z
\}
\]
is a subgroup of $\Gamma_n$ which is isomorphic to $\Gamma_{n-1}\times \mathbb Z$. Intersecting with $\Gamma$, we obtain a finite index subgroup $B_n<B_n'$, which also acts cocompactly by isometry on $X_{n-1}\times \mathbb R$. Thus the inclusion $X_{n-1}\times \mathbb R\subset X_n$ induces on the quotient a totally geodesic embedding from $B_n\backslash (X_{n-1}\times \mathbb R)$ to $\Gamma\backslash X_n$. The image is an isometric copy of $N\times S^1$ where $N$ is a closed locally symmetric manifold modeled on $\SL_{n-1}(\mathbb R)$. We denote the image by $H$.


\subsection{Transverse intersections}\label{sec:transverse} One way to show a submanifold $H$ on $M$ represents a non-trivial homology class is to find a transverse submanifold $T$ of the complementary dimension, such that the intersection number $i(H,T)\neq 0$. 

\subsection*{On the universal cover} As a first step, we would like to realize this in the universal cover so that the two lifts $\widetilde{H}$ and $\widetilde{T}$ intersect transversely. We have seen in Section \ref{sec:cycles} two totally geodesic copies $F$ and $X_{n-1}\times \mathbb R$ whose dimensions add up to $\dim(X_n)$, but according to the construction, $F$ is contained in $X_{n-1}\times \mathbb R$. We are going to take a suitable rational conjugate $F'$ of $F$ so that $F'$ intersects $X_{n-1}\times \mathbb R$ transversely, as indicated by the following picture.

\begin{figure}[ht]
		\centering
		\includegraphics[width=0.7\linewidth]{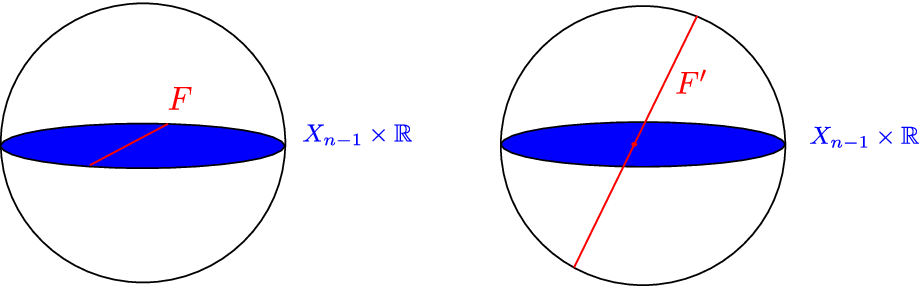}
		\caption{Transverse intersection on the universal cover}
\end{figure}

To see this, we fix a basepoint $p\in X_n$ that corresponds to $eK\in G/K$. The Cartan decomposition $\mathfrak g=\mathfrak k+\mathfrak p$ identifies the tangent space $T_p X_n$ with $\mathfrak p=\{A\in \mathfrak{sl}_n(\mathbb R):A^t=A\}$. Denote $\mathfrak a$ and $\mathfrak p_1$ the tangent subspaces of $T_p(F)$ and $T_p(X_{n-1}\times \mathbb R)$ as described above. Then $\mathfrak a$ consists of diagonal matrices with trace zero, and $\mathfrak p_1=\{A\in \mathfrak p:[A,u]=0\}$ where 
\[u= \left(\begin{array}{c|c}
 I^{(n-1)} & 0 \\
 \hline 
 0 &  -(n-1) \\ 
\end{array}\right)
\]
is a tangent vector corresponding to the singular $\mathbb R$-factor. We show that we can always move $\mathfrak a$ off $\mathfrak p_1$ by an isometry in $K$.

\begin{lemma} There exists $k\in K$ such that $\operatorname{Ad}(k)\mathfrak a\cap \mathfrak p_1=0$.
\end{lemma}

\begin{proof} We instead show there exists $k\in K$ such that $\operatorname{Ad}(k)\mathfrak p_1\cap \mathfrak a=0$. It is clear that
\[\operatorname{Ad}(k)\mathfrak p_1=\{A\in \mathfrak p:[A,kuk^{-1}]=0\}.\]
Suppose there is a non-zero matrix $v\in \operatorname{Ad}(k)\mathfrak p_1\cap  \mathfrak a$. Since $v$ is diagonal of zero trace, any matrix that commutes with $v$ need to have some zero on the off-diagonal entries. We show that we can choose a suitable $k\in \SO(n)$ such that $kuk^{-1}$ has no zero entry on the off-diagonal, which leads to a contradiction hence $\operatorname{Ad}(k)\mathfrak p_1\cap \mathfrak a=0$. We write $k$ as a block matrix
\[k= \left(\begin{array}{c|c}
O_{11} & O_{12} \\
 \hline 
 O_{21} &  O_{22} \\ 
\end{array}\right)
\]
so that $O_{11}$ is an $(n-1)\times (n-1)$ matrix and $O_{22}$ is a scalar. We compute
\[kuk^{-1}=kI^{(n)}k^{-1}-k\left(\begin{array}{c|c}
0^{(n-1)} & 0 \\
 \hline 
 0 &  n \\ 
\end{array}\right)k^{-1}=\left(\begin{array}{c|c}
I-nO_{12}O_{12}^t & -nO_{22}O_{12}\\
 \hline 
 -nO_{22}O_{12}^t &  1-nO_{22}^2 \\ 
\end{array}\right)\]
Thus we just need to pick any orthogonal matrix $k$ whose last column has no zero entry, and this certainly can be done. 
\end{proof}

If we apply the exponential map at $p$ to $\operatorname{Ad}(k)\mathfrak a$ and $\mathfrak p_1$, we obtain two totally geodesic submanifolds $F'$ and $X_{n-1}\times \mathbb R$ that intersect transversely at $p$, where $F'=kF$. The arithmetic construction of $\Gamma$ arises from the rational structure of $G$, and by \cite[Corollary 2.18]{Ben09}, we see that the subgroup
\[G_{\mathbb Q} =\{g \in \SL_n(L) \;|\;\tau(g^T)D_n g = D_n\}\]
corresponds to the rational points of $\SL_n(\mathbb R)\times \SU(n,\mathbb R)$, thus is dense in $\SL_n(\mathbb R)$. Since the result $gF$ intersects $X_{n-1}\times \mathbb R$ transversely is an open condition for $g\in G$, by the denseness there exists $g_0\in 
G_{\mathbb Q}$ such that $g_0F$ and $X_{n-1}\times \mathbb R$ intersect transversely. Moreover, since $g_0$ is rational, we see that $g_0\Gamma g_0^{-1}\cap \Gamma$ is a finite index subgroup of $\Gamma$. This implies that $g_0 A_n g_0^{-1}\cap \Gamma$ is a finite index subgroup of $g_0 A_n g_0^{-1}$, thus isomorphic to $\mathbb Z^{n-1}$. We denote $\overline{A_n}=g_0 A_n g_0^{-1}\cap \Gamma$, and it acts cocompactly on $F'$. The quotient manifold $\overline{A_n}\backslash F'$ is a totally geodesic embedded $(n-1)$-torus in $\Gamma\backslash X_n$, which we denote by $T'$.


\subsection*{On the quotient manifold} We have now obtained two totally geodesic submanifolds $T'$ and $H$ on $M$, such that they arise as the quotient of two transverse intersecting submanifolds $F'$ and $X_{n-1}\times \mathbb R$ on $X_n$. For simplicity, we still denote $T',F'$ by $T,F$. We denote the intersecting of $F'$ and $X_{n-1}\times \mathbb R$ by $\widetilde{x_0}$.

Passing to the quotient, we see that $T$ and $H$ intersect at least at one point $x_0$ transversely, but $T$ and $H$ might have other intersections (which we call \emph{extra intersections}) that might not even be transverse. Following \cite{ANP15}, we will use a double coset system to describe them.

\begin{figure}[ht]
		\centering
		\includegraphics[width=0.7\linewidth]{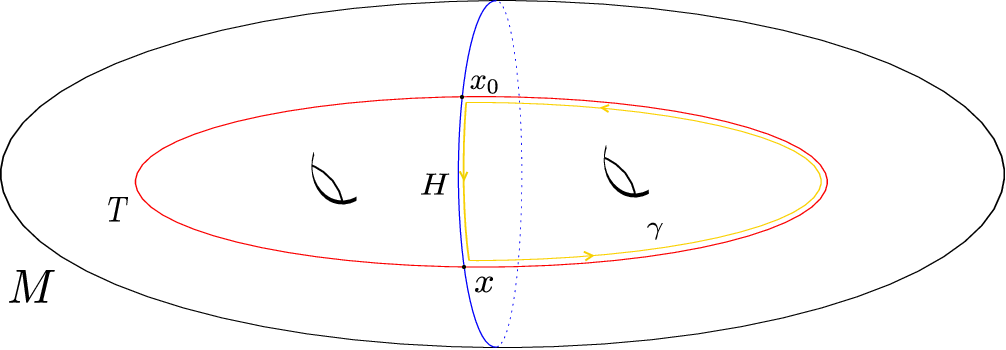}
		\caption{Double coset}
\end{figure}

Since both $T$ and $H$ are totally geodesic, the connected components of $T\cap H$ are also closed totally geodesic submanifolds. For each intersection $x\neq x_0$, we can make a loop at $x_0$ by first connecting $x_0$ to $x$ along $H$ and then travel back from $x$ to $x_0$ along $T$. This gives an element $\gamma\in \pi_1(M,x_0)\cong \Gamma$. But for a different choice of path connecting $x_0$ to $x$ along $H$, they differ by a left multiplication of an element in $\pi_1(x_0, H)$, and similarly, for a different choice of path connecting back from $x$ to $x_0$ along $T$, they differ by a right multiplication of an element in $\pi_1(x_0, T)$. Therefore, for each intersection $x$, it corresponds to a double coset
\[\pi_1(H)\cdot\gamma\cdot\pi_1(T),\]
which we simply denote by $[\gamma]$. It is clear that, if $x,x'$ belong to the same connected component of $T\cap H$, then the corresponding double cosets are the same. On the other hand, if $[\gamma]$ represents the connected component of an intersection $x$, then $\gamma \widetilde T\cap \widetilde H$ projects to $x$. Thus, if $x,x'$ belong to different components, then the corresponding double cosets are different.

\begin{lemma}\label{lem:finite} There are only finitely many connected components of $T\cap H$.
\end{lemma}

\begin{proof}
Since both $H$ and $T$ are compact, there exists a constant $D>0$ that $\diam(H)<D$ and $\diam(T)<D$. Therefore, in view of the above construction, for any extra intersection $x$, we can always choose a representing loop such that the two connecting paths between $x_0$ and $x$ are both $< D/2$, so the total length of the representing loop is $<D$. Thus, any connected component of $T\cap H$ corresponds to a double coset $[\gamma]$ where $\gamma$ can be chosen in a finite set $\{\gamma\colon d(\widetilde{x_0},\gamma \widetilde{x_0})<D\}$. This completes the proof.
\end{proof}


\subsection{Sign of intersections} Extra intersections of $T$ and $H$ might not be transverse, and even if they do, they might take different sign which can possibly add up to a zero intersection number. We will need the following criterion which assures that the intersection is transverse and takes the same sign as the base intersection $x_0$.

\begin{lemma}\cite[Claim in Section 8]{ANP15} (See also \cite[Lemma 2]{ANP20})\label{lem:sign}
If an extra intersection $x$ has a representative (thus for all representatives) $\gamma$ which can be written as a product of two elements in $\SL_n(\mathbb R)$,
\[\gamma=a\cdot b\]
such that $a$ preserves $\widetilde{H}$ and its orientation and $b$ preserves $\widetilde{T}$ and its orientation, then $x$ is a transverse intersection and takes the same sign as $x_0$.
\end{lemma}

\begin{proof} For the completeness, we include a proof here. Lifting to the universal cover, the intersection $\gamma\widetilde{T}\cap \widetilde{H}$ is a lift of $x$, which we denote by $\widetilde{x}$. Since $\gamma=ab$, and $a,b$ preserves $\widetilde{H}, \widetilde{T}$ respectively, we can write
\[\widetilde{x}=\gamma\widetilde{T}\cap \widetilde{H}=ab\widetilde{T}\cap \widetilde{H}=a(\widetilde{T}\cap a^{-1}\widetilde H)=a(\widetilde{T}\cap\widetilde H)=a\cdot \widetilde{x_0}.\]
Thus, $a$ sends the local intersecting configuration $\widetilde{T}\cap\widetilde H$ around $\widetilde{x_0}$ orientation preserving isometrically to $\gamma\widetilde{T}\cap \widetilde{H}$ around $\widetilde{x}$. In particular, $\widetilde{x}$ is a transverse intersection and takes the same sign as $\widetilde{x_0}$. The same holds when projects down to $M$.
\end{proof}

\subsection{Interpreting as linear equations} If $[\gamma]$ represents an extra intersection, we are interested to know when the criterion of Lemma \ref{lem:sign} holds.
\begin{lemma}\label{lem:criterion} The criterion of Lemma \ref{lem:sign} holds if and only if the following equations can be solved for $a, b\in \operatorname{Mat}_n(\mathbb R)$,
\begin{enumerate}
\item[(1)] $\gamma=a\cdot b$,
\item[(2)] $b\cdot t=t\cdot b$,
\item[(3)] $a\cdot u=u\cdot a$,
\item[(4)] $\det a=1$,
\item[(5)] $a$ preserves the orientation of $\widetilde H$,
\end{enumerate}
where $t\in \Isom(\widetilde T)$ is any isometry with distinct eigenvalues, and $u=\diag(1,...,1,-(n-1))$ as in Section \ref{sec:transverse}.
\end{lemma}

\begin{proof}
Notice that $(2)$ holds if and only if $b$ preserves $\widetilde T$, and $(3)$ holds if and only if $a$ preserves $\widetilde{H}$, which is also equivalent to $a$ being certain block diagonal matrix. Thus, it is clear the above equations are necessary. To see they are also sufficient, we need to show $\det b=1$ and $b$ preserves the orientation of $\widetilde T$. Since $\gamma\in \Gamma<\SL_n(\mathbb R)$, $(1)$ and $(4)$ imply $\det b=1$. Furthermore, $\det b=1$ and $(2)$ imply $b$ is orientation preserving on $\widetilde T$.
\end{proof}

Now if we ignore for the moment $(4)$ and $(5)$, and choose $t$ as a rational matrix, then the above equations become a system of \emph{linear} equations (on variable $a$) whose coefficients are in $L=\mathbb Q[\sqrt[4]2]$, after rewriting in the following way:
\begin{equation}\tag{$\ast$}\label{eq:linear}
\begin{array}{cc}
\begin{cases} 
      \gamma = a \cdot b \\
      b \cdot t = t \cdot b\\
      a\cdot u=u\cdot a\\
   \end{cases}
 & \implies \begin{cases} 
      a \cdot t = \gamma  t \gamma^{-1} \cdot a \\
      a\cdot u=u\cdot a\\
   \end{cases}
\end{array} 
\end{equation}


\section{Proof of theorem}\label{sec:proof} In this section, we prove Theorem \ref{thm:main}. The approach is very similar to \cite{ANP20}. Following the notation in Section \ref{sec:construction}, let $M=\Gamma\backslash X_n$ be the closed locally symmetric manifold and $T, H\subset M$ be the closed submanifolds that we constructed. Although $H$ might represent a trivial homology class in $M$, we want to show that under a sufficiently large degree cover $M'$ of $M$, the lift of $H$ will represent a non-trivial class in $M'$. As a first step, we need to understand how extra intersections behave under congruence covers.

\subsection{Extra intersections under congruence covers} For each prime ideal $p\subset \mathcal O_L$, and positive integer $k$, we denote the congruence subgroup $\Gamma(p^k)=\Gamma\cap \ker{\varphi_{p^k}}$, where $\varphi_{p^k}$ is the (mod $p^k$)-homomorphism as in Section \ref{sec:lattice}. For any extra intersection $[\gamma] = \pi_1(H)\cdot \gamma\cdot \pi_1(T),$
by passing onto the (mod $p^k$) congruence cover, either the intersection gets removed, which then we do not need to worry about anymore, or the intersection remains, in which case it has to satisfy certain algebraic constraints.

Assume an extra intersection $x$ remains under a (mod $p^k$) congruence cover $M'$ of $M$, then there is a lift $T'$ of $T$ and $H'$ of $H$ such that there is an intersection of $T'$ and $H'$ in $M'$ that projects to $x$. This means $[\gamma]$ has a representative in $\pi_1(M')\cong \Gamma(p^k)$. Therefore, there exist $a \in \pi_1(H)$ and $b \in \pi_1(T)$ such that 
$$ a^{-1}\cdot\gamma\cdot b^{-1} \in \pi_1(M')<\ker(\varphi_{p^k}),$$
which implies that $\gamma = a\cdot b$ (mod $p^k$). Since an element $a\in \pi_1(H)$ commutes with $u$, and an element $b \in \pi_1(T)$ commutes with $t$. We obtain the following proposition.

\begin{lemma} If an extra intersection $[\gamma]$ remains under a (mod $p^k$) congruence cover, then the linear system of equations \eqref{eq:linear}
\[
\begin{cases} 
      a \cdot t = \gamma  t \gamma^{-1} \cdot a \\
      a\cdot u=u\cdot a\\
   \end{cases}
\]
has a (mod $p^k$) solution.
\end{lemma}

\subsection{Solving the linear system}\label{sec:linear-eq} However, some extra intersections might survive all finite covers, which we call \emph{left over} intersections. Thus by the above lemma, if $[\gamma]$ represents a left over intersection, then for all prime ideal $p$ and any $k\in \mathbb Z^+$, the linear system of equations \eqref{eq:linear} 
\[
\begin{cases} 
      a \cdot t = \gamma  t \gamma^{-1} \cdot a \\
      a\cdot u=u\cdot a\\
   \end{cases}
\]
has a (mod $p^k$) solution $a_k$. Fix $p$ and let $k\rightarrow \infty$, we obtain a $p$-adic solution $a_p=\lim_{k\rightarrow \infty} a_k$ in the $p$-adic completion $L_p$ of $L$. Since each $a_k\in \SL_{n}(\mathcal O_L)$, we conclude that $a_p\in \SL_{n}(\mathbb Z[\sqrt[4]2]_p)$. Moreover, the linear system of equations is defined over $L$, so if $V$ is the solution space over the field $L$, then the solution space over $L_p$ is exactly $V\otimes L_p$. In particular, the fact that the linear system has a solution $a_p\in L_p$ implies that the system is consistent, hence it also has a solution over the base field $L$. Since $\det a_p=1\neq 0$ and because $L$ is dense in $L_p$, we can choose the solution over $L$ to be invertible.

\begin{lemma}\label{lem:one-dim} If $[\gamma]$ represents a left over intersection, then the solution space to the linear system of equations
\[
\begin{cases} 
      a \cdot t = \gamma  t \gamma^{-1} \cdot a \\
      a\cdot u=u\cdot a\\
   \end{cases}
\]
is $1$-dimensional.
\end{lemma}

\begin{proof} By the above discussion, we know that the solution space is at least $1$-dimensional, which is spanned by an invertible matrix $a\in \GL_n(L)$. Suppose there is another solution $\bar a\in \operatorname{Mat}_n(L)$, then by substituting $a,\bar a$ into the equations, we obtain
\[
\begin{cases} 
      (a^{-1}\bar a)\cdot t=t\cdot (a^{-1}\bar a) \\
      (a^{-1}\bar a)\cdot u=u\cdot (a^{-1}\bar a)
   \end{cases},
\]
or equivalently, $(a^{-1}\bar a)\in Z(t)\cap Z(u)$, where the centralizers are taken in $\operatorname{Mat}_n(L)$. Now we make the field extension $L\subset \mathbb R$, and consider the linear system over the real numbers. Since $Z(t)\cap Z(u)$ is a linear subspace (considered over $\mathbb R$) containing at least a one dimensional subspace spanned by $I$, if it has more than one dimension, then we can find a matrix $s$, linearly independent with $I$, such that $s$ commutes with both $t$ and $u$. Then there is a curve of isometries $g_\epsilon\in \SL_n(\mathbb R)$ around $I$ given by
\[g_\epsilon=\frac{I+\epsilon s}{\sqrt[n]{\det (I+\epsilon s)}},\quad |\epsilon|<\epsilon_0,\]
where $\epsilon_0$ is a sufficiently small real number such that $\det(I+\epsilon s)$ stays positive. However, any isometry $g\in \SL_n(\mathbb R)$ which commutes with both $t$ and $u$ has to preserve both $\widetilde T$ and $\widetilde H$ hence also their intersection $\widetilde{x_0}$, so it must be in $\Stab_{G}\widetilde{x_0}\cong \SO(n)$. Since there is only finitely many elements in $K$ that commutes with $t$ (because $t$ is in the regular direction), the continuous family $g_\epsilon$ must be constant, hence $s$ is a scalar multiple of $I$, which is a contradiction. This implies $Z(t)\cap Z(u)$ is one dimensional spanned by $I$, and that $a,\bar a$ can only differ by a scalar. In other words, the solution space is one dimensional.

\end{proof}


\subsection{Finishing the proof} By Lemma \ref{lem:finite}, there are only finitely many connected components of $T\cap H$. For each component, either it goes away when lifting to a finite cover, or it is a left over intersection. Thus when passing to a sufficiently large degree cover $M'$ of $M$, only the left over intersections remain in $T'\cap H'$. To obtain $i(T',H')\neq 0$, it is sufficient to show the following proposition.

\begin{proposition} If $[\gamma]$ represents a left over intersection, then it is transverse and takes the same sign as base intersection $x_0$.
\end{proposition}

\begin{proof}
In view of Lemma \ref{lem:sign} and \ref{lem:criterion}, we just need to find a solution $a, b\in \operatorname{Mat}_n(\mathbb R)$ satisfying $(1)-(5)$ in Lemma \ref{lem:criterion}. From Section \ref{sec:linear-eq}, we already have a solution $a\in \operatorname{Mat}_n(L)$ with non-zero determinant, thus we can solve for $b=a^{-1}\cdot \gamma$ and $(1)-(3)$ follows. The rest of the proof is to modify $a$ to satisfy $(4)$ and $(5)$.

For $(4)$, we compare the rational solution $a\in \operatorname{Mat}_n(L)$ with the $p$-adic solutions $a_p\in \SL_{n}(\mathbb Z[\sqrt[4]2]_p)$ for each prime ideals $p$. Since $\det a\neq 0$, by Lemma \ref{lem:one-dim} there exists a non-zero $c_p\in L_p$ such that $a=c_p\cdot a_p$. Take the determinant on both sides, we obtain $\det a=c_p^{n}\det a_p=c_p^n$, and this holds for all prime ideal $p$. It follows that the equation $x^n=\det a$ is solved locally in $L_p$ for all $p$, thus by Corollary \ref{cor:Grunwald-Wang}, it is also solved globally in $L$. In other words, there exists a non-zero $c\in L$ such that $c^n=\det a$. Therefore, if we set $\bar a=a/c$, and correspondingly $\bar b=c\cdot b$, then $\bar a, \bar b$ satisfies $(1)-(4)$.

We need to show $\bar a$ also satisfies $(5)$. Note that $(3)$ implies that $\bar a$ is a block diagonal matrix whose lower-right entry forms a single block. To check whether $\bar a$ is orientation preserving on $\widetilde H$, we need the following lemma.

\begin{lemma}\label{lem:orientation} For the block diagonal matrix $\bar a$, it preserves the orientation of $\widetilde{H}=X_{n-1}\times \mathbb R$ exactly when either $n$ is even, or $n$ is odd and the lower-right entry of $\bar a$ is positive.
\end{lemma}
\begin{proof}
we write
\[\bar a=\left(\begin{array}{cc}
A_0 & 0 \\ 
0& z
\end{array}\right). \] 
If $n$ is even, then $\bar a$ and $-\bar a$ acts the same on $\widetilde{H}$ hence we can always assume $z>0$, and consequently $\det A_0>0$ since $\det A_0\cdot z=1$. Then $\bar a$ is line homotopic to
\[\bar a'=\left(\begin{array}{cc}
z^\frac{1}{n-1}A_0 & 0 \\ 
0& 1
\end{array}\right) \]
inside $\Isom \widetilde{H}$, where $z^\frac{1}{n-1}A_0\in \SL_{n-1}(\mathbb R)$ is orientation preserving on $X_{n-1}$. Thus $\bar a'$ is orientation preserving on $\widetilde H$. Since under continuous variations the orientation does not change, we conclude $\bar a$ is also orientation preserving on $\widetilde{H}$.

Similarly, if $n$ is odd, and $z>0$, the upper left block of $\bar a$ also has positive determinant. Hence by the same argument in the above, we have $\bar a$ is orientation preserving. If $n$ is odd and $z<0$, then it differs from an orientation preserving isometry by
$$\left(\begin{array}{ccc}
-1 & &  \\ 
 & I^{(n-2)} &  \\ 
 & & -1
\end{array} \right).$$
One can check that the adjoint action of the above matrix on $\widetilde H$ is orientation reversing. This implies that in this case $\bar a$ is orientation reversing.
\end{proof}

Now we continue with the proof of Theorem \ref{thm:main}. We have $\bar a =a/c=(c_p/c)\cdot a_p$. Take determinant on both sides, we obtain $(c_p/c)^n=1$. Thus $c_p/c$ is a $p$-adic $n$-th root of unity and in particular it is a $p$-adic integer, that is, $c_p/c\in \mathbb{Z}[\sqrt[4]{2}]_p$. We denote $\omega_p=c_p/c$, and we have $\bar a=\omega_p \cdot a_p\in  \SL_n(\mathbb{Z}[\sqrt[4]{2}]_p)\cap \SL_n(L)=\SL_n(\mathcal O_L)$ so that the lower-right entry $z\in \mathbb{Z}[\sqrt[4]{2}]$. According to Lemma \ref{lem:orientation}, we need to show $z>0$ in the case $n$ is odd. So from now on, we assume $n$ is odd.

First, we claim $z$ is a unit in $L$ which satisfies $\tau(z)z=1$. Since $a_p$ arises as a $p$-adic limit of $a_k$, it is also block diagonal. We denote the lower-right entry of each $a_k$ by $z_k$, and the lower-right entry of $a_p$ by $z_p$. Since each $a_k$ are elements in $\Gamma$, by construction, we have $\tau(z_k)z_k=1$. Thus passing to the limit, we have $\tau(z_p)z_p=1$. Comparing the lower-right entries of equation $\bar a =\omega_p\cdot a_p$, we have $z=\omega_p\cdot z_p$. So we compute $\tau(z)z=\tau(\omega_p)\omega_p\tau(z_p)z_p=\tau(\omega_p)\omega_p\in \mathcal O_L$ since $z\in \mathcal O_L$. But $\tau(\omega_p)\omega_p$ is also a $p$-adic $n$-th root of unity, so it must be $\pm 1$. Because $n$ is odd, it can only be $+1$ and therefore, $\tau (z)z =1$.

Next, we prove $z>0$. By Lemma \ref{lem:fundamental unit}, we have $z = \pm u_0^m$ for some integer $m$. Recall that $a_p$ is a $p$-adic limit of $a_k \in \SL_n(\mathcal O_L)$, and since each $a_k$ preserves the orientation of $\widetilde{H}$, by Lemma \ref{lem:orientation}, the lower-right entry of each $a_k$ is positive hence $z_k$ has the form $+u_0^{m_k}$. Therefore,
\[ z = \lim_{k\rightarrow\infty} (\omega_p \cdot u_0^{m_k}).\]
Suppose for the purpose of contradiction that $z = -u_0^m$, then we have
\[ -1 = \lim_{k\rightarrow\infty} (\omega_p \cdot u_0^{m_k-m}).\] 
Taking the $n$-th power on both side (and note $n$ is odd), we obtain 
\[ -1  = \lim_{k\rightarrow\infty} u_0^{n(m_k-m)}.\]
Note that the above equation is in the sense of $p$-adic limit. In particular, when $p=\sqrt[4]2$ it means that $-1$ is the $\sqrt[4]2$-adic limit of a sequence of powers of $u_0$. By possibly passing to a subsequence and possibly raising to power $-1$ on both sides, we can assume it is a sequence of positive powers of $u_0$. We claim this is impossible.

Assume on the contrary that $-1=\lim_{i\rightarrow \infty} u_0^{k_i}$ for some positive sequence $k_i$ under the $\sqrt[4]2$-adic metric. Since $4=(\sqrt[4]2)^8$, we have $-1=\lim_{i\rightarrow \infty} u_0^{k_i}$ (mod 4). However, by direct computations we have
\[u_0^2=1\;(\textrm{mod } 4),\quad u_0 \ne -1 (\textrm{mod } 4).\]
Thus $u_0^k\neq -1$ (mod 4) for any integer $k$. In particular $-1=\lim_{i\rightarrow \infty} u_0^{k_i}$ (mod 4) cannot hold. The contradiction shows that $z$ can only be of the form $+u_0^m$, so $z>0$.

We have then verified $\bar a, \bar b$ satisfy all conditions $(1)-(5)$ in Lemma \ref{lem:criterion}, and therefore by Lemma \ref{lem:sign}, the left over intersection of $[\gamma]$ is transverse and has the same sign as the base intersection. This completes the proof.
\end{proof}


\section{Gromov norm and bounded cohomology}\label{sec:Gromov norm}

\subsection{Gromov norm}
We have constructed above a closed locally symmetric manifold $M$ modeled on $\SL_n(\mathbb R)/\SO(n)$, and a homology class $\alpha\in H_k(M,\mathbb R)$ represented by a totally geodesic submanifold which has a circle factor. Thus it is not hard to see that $\alpha$ has zero Gromov norm, and this proves Corollary \ref{cor:main}.

\begin{proposition}\label{cor:S1-factor}
If $\alpha\in H_k(M,\mathbb R)$ is represented by a submanifold $N\times S^1\subset M$, then $||\alpha||_1=0$.
\end{proposition}
\begin{proof}
Let $i:N\times S^1\rightarrow M$ be the inclusion map, then by the properties of Gromov norm \cite{Gro82}, we have
$$||\alpha||_1=||i_*([N\times S^1])||_1\leq ||N\times S^1||\leq C||N||\cdot||S^1||=0.$$
\end{proof}

\subsection{Dupont's conjecture}

Dual to the $\ell^1$ norm on the homology, one can rephrase results on Gromov norm using bounded cohomology. By the duality principle \cite{Gro82} (See also \cite[Proposition F.2.2]{BP92}), if $||\alpha||= 0$, then any cohomology class $\omega\in H^k(M,\mathbb R)$ satisfying $\langle \omega, \alpha\rangle\neq 0$ does not have bounded representatives. Thus, producing a non-trivial zero Gromov norm homology class is equivalent to the non-surjectivity of the comparison map in the same degree. On the other hand, it is conjectured in \cite{Dup79} that all $G$-invariant forms on $X$ (arising from the continuous cohomology $H_c^*(G,\mathbb R)$ via the van Est isomorphism) have bounded representatives.

\begin{conj}(Dupont) If $G$ is a connected, non-compact, semi-simple Lie group with finite center, then the comparison map
$$c_G^*:H_{b,c}^*(G,\mathbb R)\rightarrow H_{c}^*(G,\mathbb R)$$
is always surjective.
\end{conj}

\begin{remark} Monod further conjectured \cite{Mon06} that this is an isomorphism.
\end{remark}

For any cocompact lattice $\Gamma<G$, we have the following commutative diagram
\[\begin{tikzcd}
H_{b,c}^k(G,\mathbb R) \arrow[r, "i_b^*"] \arrow[d,"c_G^k"]
& H_b^k(\Gamma,\mathbb R) \arrow[d,"c_\Gamma^k"] \\
H_c^k(G,\mathbb R) \arrow[r, "i^*"]
& H^k(\Gamma,\mathbb R)
\end{tikzcd}
\]
where $i^*,i_b^*$ are the natural homomorphisms induced by the inclusion $i:\Gamma\rightarrow G$, and $c_G^k,c_\Gamma^k$ are the comparison maps. The van Est isomorphism \cite{Van55} states that the continuous cohomology $H_c^*(G,\mathbb R)$ is isomorphic to the relative Lie algebra cohomology $H^*(\mathfrak{g},\mathfrak{k},\mathbb R)$, and any such class can be identified with a $K$-invariant alternating form on $\mathfrak{g}/\mathfrak{k}\cong \Omega^*(T_pX)^K$. By left translation, it further extends to a $G$-invariant form on the symmetric space $G/K$. Therefore, under this identification $i^*$ can be viewed as the restriction map from the $G$-invariant forms to the $\Gamma$-invariant forms on $G/K$.

Now specify $G=\SL_n(\mathbb R)$, $\Gamma$ the cocompact lattice as in Corollary \ref{cor:main}, and $\alpha\in H_k(\Gamma)$ the non-trivial homology class with zero Gromov norm. If a class $\eta\in H_c^k(G,\mathbb R)$ satisfies $\langle i^*(\eta), \alpha\rangle\neq 0$, then it is not in the image of $c_\Gamma^k$, hence neither is in the image of $c_\Gamma^k\circ i_b^*=i^*\circ c_G^k$. This implies that $c_G^k$ is not surjective, giving a counterexample to the Dupont's conjecture. In other words, if the Dupont's conjecture is true, then $||\alpha||_1=0$ should imply $\int_{\alpha}i^*(\eta)=0$. The following proposition confirms this, so our construction of zero Gromov norm class does not violate with Dupont's conjecture.

\begin{proposition}
Let $G=\SL_n(\mathbb R)$, $\Gamma$ be any cocompact lattice in $G$, $X_n=G/K$ and $X_{n-1}$ be the totally geodesic submanifold corresponding to the block diagonal Lie subgroup $\SL_{n-1}(\mathbb R)\times \mathbb R$.
If $\eta$ is any $G$-invariant differential $k$-form on $X_n$, where $k=\dim(X_n)-\rank(X_n)$, and $\alpha\in H_k(\Gamma\backslash X_n)$ is any class represented by a totally geodesic submanifold covered by $X_{n-1}\times \mathbb R$, then
$$\int_{\alpha}\eta=0.$$ 
\end{proposition}

\begin{proof} Since the integral is homogeneous, it is sufficient to show that at any basepoint $p$, the evaluation of a $K$-invariant form $\eta$ on the tangent subspace $T_p(X_{n-1}\times \mathbb R)\subset T_pX_n$ is always zero.

We identify $T_pX_n$ with $\mathfrak p$ using the Cartan decomposition $\mathfrak g\cong \mathfrak k+\mathfrak p$ at $p$, and denote $\theta:\mathfrak g\rightarrow \mathfrak g$ the Cartan involution. Fix a maximal abelian subgroup $\mathfrak a$ in $\mathfrak p$ (unique up to conjugate), under the root space decomposition, we have $\mathfrak{g}\cong \mathfrak{g}_0+\sum_{\lambda\in \Lambda}\mathfrak{g}_{\lambda}$ where $\Lambda$ denotes the set of all roots of $\mathfrak g$. Since $\theta$ sends $\mathfrak g_\lambda$ to $\mathfrak g_{-\lambda}$, if we denote $\mathfrak k_{\lambda}=(I-\theta)\mathfrak g_{\lambda}\subset \mathfrak k$ and $\mathfrak p_{\lambda}=(I+\theta)\mathfrak g_{\lambda}\subset \mathfrak p$, then we obtain the root space decompositions for $\mathfrak k$ and $\mathfrak p$: $\mathfrak{k}\cong \mathfrak{k}_0+\sum_{\lambda\in \Lambda^+}\mathfrak{k}_{\lambda}$ and $\mathfrak{p}\cong \mathfrak{a}+\sum_{\lambda\in \Lambda^+}\mathfrak{p}_{\lambda}$, where $\Lambda^+$ is the set of all positive roots and the latter splits orthogonally with respect to the $G$-invariant metric (Killing form) on $X_n$.

In the case $\mathfrak g=\sl_n(\mathbb R)$, the above decompositions have the following explicit forms. Denote $E_{ij}\in \mathfrak g$ the $n\times n$ matrix whose $(i,j)$-entry is $1$ and $0$ elsewhere. Choose the maximal abelian subgroup $\mathfrak a\subset \mathfrak p$ to be the subset of diagonal matrices, then the roots of $\mathfrak g$ are exactly dual to $\{E_{ii}-E_{jj}\}$ for each $i\neq j$. Under the root space decompositions, the invariant spaces of the root $\lambda=E_{ii}^*-E_{jj}^*$ are given by the one dimensional spaces
$$\mathfrak{g}_\lambda=\langle E_{ij}\rangle,\quad\mathfrak{k}_\lambda=\langle E_{ij}-E_{ji}\rangle,\quad\mathfrak{p}_\lambda=\langle E_{ij}+E_{ji}\rangle,$$
and for the convenience, we denote by $\mathfrak{g}_{ij}, \mathfrak{k}_{ij}$ and $\mathfrak{p}_{ij}$ respectively. Note also that $\mathfrak k_0=0$ and $\mathfrak g_0=\mathfrak a$.

We pick an orthonormal basis $e_1,...,e_{n-1}$ in $\mathfrak a$ and $e_{ij}$ the normalized unit vector of $E_{ij}+E_{ji}$ in $\mathfrak p_{ij}$. So the collection of $\{e_i|_{i=1}^{n-1},e_{ij}|_{1\leq i<j\leq n}\}$ form an orthonormal basis in $\mathfrak p$. Now the tangent space of the totally geodesic submanifold $X_{n-1}\times \mathbb R$ is spanned by $\{\mathfrak a, \mathfrak p_{ij}\}$ where $1\leq i<j<n$, so its normal space is spanned by $\{e_{in}|_{i=1}^{n-1}\}$. For any alternating $k$-form $\eta$ on $T_p(X_{n-1}\times \mathbb R)$, we can write in terms of the $k$-wedge of the orthonormal basis, and the evaluation of $\eta$ on $T_p(X_{n-1}\times \mathbb R)$ is (up to a sign) the coefficient in front of
$$e_1^*\wedge...\wedge e_{n-1}^*\wedge\bigwedge_{1\leq i<j<n} e_{ij}^*,$$
which is the same as the evaluation of $\star \eta$ (Hodge star operator) on the normal space of $T_p(X_{n-1}\times \mathbb R)\subset T_pX_n$. Our goal is to show for any $K$-invariant form $\eta$ this value is zero. Since the Hodge star operator sends $K$-invariant forms to $K$-invariant forms, it is equivalent to show that for any $K$-invariant $(n-1)$-form, the coefficient in front of $e_{1n}^*\wedge...\wedge e_{(n-1)n}^*$ is zero.
Since $\star \eta$ is $\operatorname{Ad}(K)$-invariant, it is $\operatorname{ad}(\mathfrak k)$-vanishing. In particular, for $u=E_{1n}-E_{n1}\in \mathfrak k_{1n}$ we have $\operatorname{ad}(u)(\star \eta)=0$. However, for the linear endomorphism $\operatorname{ad}(u):\bigwedge^{n-1}\mathfrak p\rightarrow \bigwedge^{n-1}\mathfrak p$, there is an invariant $2$-dimensional subspace $V_0$ spanned by
$$\{e_{1n}^*\wedge e_{2n}^*\wedge...\wedge e_{(n-1)n}^*, H^*\wedge e_{12}^*\wedge...\wedge e_{1(n-1)}^*\}$$
where $H=E_{11}-E_{nn}$. This is because by the Lie algebra computations, 
$$\operatorname{ad}(u)(e_{in})=e_{1i},\quad \operatorname{ad}(u)(e_{1i})=-e_{in},\quad 1<i<n,$$
and 
$$\operatorname{ad}(u)(e_{1n})=a\cdot H,\quad \operatorname{ad}(u)(H)=b\cdot e_{1n},\quad a,b\neq 0,$$
where $a,b$ depends on the scale of the metric. Thus the restriction of the linear endomorphism $\operatorname{ad}(u)|_{V_0}$ has matrix form
$$\left(\begin{array}{cc}
0 & a \\ 
(-1)^nb & 0
\end{array}\right) $$
which is non-singular. Therefore, $\operatorname{ad}(u)(\star \eta)=0$ implies that the component of $\star \eta$ on $V_0$ must be zero, and in particular the coefficient in front of $e_{1n}^*\wedge...\wedge e_{(n-1)n}^*$ is zero. This completes the proof.

\end{proof}

\begin{remark}
	It is pointed out by the anonymous referee that one can use the stability results to show the map (induced by the inclusion)
	\[H_c^k(\SL_{n-1}(\R)\times \R)\rightarrow H_c^k(\SL_n(\R))\]
	is zero in degree $k=\dim(X_n)-\rank(X_n)$. The proposition then follows immediately. Our proof is rather an explicit computation, which might be of independent interest.
\end{remark}


\def\cprime{$'$}



\end{document}